\documentclass[11pt, a4paper]{article}
\usepackage{typearea}
\typearea{12}

\usepackage{amsmath, amssymb, enumerate, amsthm, accents}
\usepackage{amsmath, amssymb, url, comment}
\usepackage[numbers]{natbib}
\usepackage{tikz}
\usepackage{pgfplotstable}
\usetikzlibrary{arrows,positioning,plotmarks,external,patterns,angles,decorations.pathmorphing,backgrounds,fit,shapes,graphs,calc,spy}
\pgfplotsset{compat=1.14}

\theoremstyle{plain}
\newtheorem{theorem}{Theorem}[section]
\newtheorem{lemma}[theorem]{Lemma}
\newtheorem{corollary}[theorem]{Corollary}
\newtheorem{proposition}[theorem]{Proposition}
\theoremstyle{remark}
\newtheorem{definition}[theorem]{Definition}

\begin{document}
	
	\title{Matrix quadratic risk of orthogonally invariant estimators\\ for a normal mean matrix}
	\author{Takeru~Matsuda\thanks{Graduate School of Information Science and Technology, the University of Tokyo \& RIKEN Center for Brain Science, e-mail: 
			\texttt{matsuda@mist.i.u-tokyo.ac.jp}}}
	
	\date{}
	
	\maketitle

\begin{abstract}
	In estimation of a normal mean matrix under the matrix quadratic loss, we develop a general formula for the matrix quadratic risk of orthogonally invariant estimators.
	The derivation is based on several formulas for matrix derivatives of orthogonally invariant functions of matrices.
	As an application, we calculate the matrix quadratic risk of a singular value shrinkage estimator motivated by Stein's proposal for improving on the Efron--Morris estimator 50 years ago.
\end{abstract}

\section{Introduction}
Suppose that we have a matrix observation $X \in \mathbb{R}^{n \times p}$ whose entries are independent normal random variables $X_{ij} \sim {\rm N} (M_{ij},1)$, where $M \in \mathbb{R}^{n \times p}$ is an unknown mean matrix. 	
In this setting, we consider estimation of $M$ under the {matrix quadratic loss} \citep{abu,Matsuda22}:
\begin{align}
	L(M,\hat{M}) = (\hat{M} - M)^{\top} (\hat{M} - M), \label{mat_loss}
\end{align}
which takes a value in the set of $p \times p$ positive semidefinite matrices.
The risk function of an estimator $\hat{M}=\hat{M}(X)$ is defined as $R(M,\hat{M}) = {\rm E}_M [ L ( M,\hat{M}(X) ) ]$, and an estimator $\hat{M}_1$ is said to dominate another estimator $\hat{M}_2$ if $R(M,\hat{M}_1) \preceq R(M,\hat{M}_2)$ for every $M$, where $\preceq$ is the L\"{o}wner order: $A \preceq B$ means that $B-A$ is positive semidefinite.
Thus, if $\hat{M}_1$ dominates $\hat{M}_2$ under the matrix quadratic loss, then
\begin{align*}
	{\rm E}_M [ \| (\hat{M}_1 - M) c \|^2 ] \leq {\rm E}_M [ \| (\hat{M}_2 - M) c \|^2 ]
\end{align*}
for every $M$ and $c \in \mathbb{R}^p$.
In particular, each column of $\hat{M}_1$ dominates that of $\hat{M}_2$ as an estimator of the corresponding column of $M$ under quadratic loss.
Recently, \cite{Matsuda22} investigated shrinkage estimation in this setting by introducing a concept called matrix superharmonicity, which can be viewed as a generalization of the theory by Stein \cite{Stein74} for a normal mean vector.
Note that shrinkage estimation of a normal mean matrix under the Frobenius loss, which is the trace of the matrix quadratic loss, has been well studied, e.g. \cite{Matsuda,Tsukuma08,Tsukuma,Yuasa,Yuasa2,Zheng}.

Many common estimators of a normal mean matrix are orthogonally invariant. 
Namely, they satisfy $\hat{M}(PXQ) = P \hat{M}(X) Q$ for any orthogonal matrices $P \in O(n)$ and $Q \in O(p)$.
It can be viewed as a generalization of the rotationally invariance of estimators for a normal mean vector, which is satisfied by many minimax shrinkage estimators \cite{shr_book}.
We focus on orthogonally invariant estimators given by
\begin{align}
	\hat{M} = X+\widetilde{\nabla}h(X), \label{pseudo}
\end{align}
where $h(X)$ is an orthogonally invariant function satisfying $h(PXQ)=h(X)$  for any orthogonal matrices $P \in O(n)$ and $Q \in O(p)$ and $\widetilde{\nabla}$ is the matrix gradient operator defined by \eqref{grad}.
For example, the maximum likelihood estimator $\hat{M}=X$ corresponds to $h(X)=0$.
The Efron--Morris estimator \cite{Efron72} defined by $\hat{M}=X(I-(n-p-1)(X^{\top}X)^{-1})$ when $n-p-1>0$ corresponds to $h(X)={-(n-p-1)/2} \cdot \log \det (X^{\top} X)$.
This estimator can be viewed as a matrix generalization of the James--Stein estimator and it is minimax under the Frobenius loss \cite{Efron72} as well as matrix quadratic loss \cite{Matsuda22}.
We will provide another example of an orthogonally invariant estimator of the form \eqref{pseudo} in Section~\ref{sec:example}.
Note that an estimator of the form \eqref{pseudo} is called a pseudo-Bayes estimator \cite{shr_book}, because it coincides with the (generalized) Bayes estimator when $h$ is given by the logarithm of the marginal distirbution of $X$ with respect to some prior on $M$ (Tweedie's formula).

In this study, to further the theory of shrinkage estimation under the matrix quadratic loss, we develop a general formula for the matrix quadratic risk of orthogonally invariant estimators of the form \eqref{pseudo}.
First, we prepare several matrix derivative formulas in Section~\ref{sec:matrix}.
Then, we derive the formula for the matrix quadratic risk in Section~\ref{sec:main}. 
Finally, we present an example in Section~\ref{sec:example}, which is motivated by Stein's proposal for improving on the Efron--Morris estimator 50 years ago \cite{Stein74}. 

\section{Matrix derivative formulas}\label{sec:matrix}
Here, we develop matrix derivative formulas based on \cite{Stein74}.
Note that, whereas \cite{Stein74} considered a setting where $X$ is a $p \times n$ matrix, here we take $X$ to be a $n \times p$ matrix.
In the following, the subscripts $a$, $b$, $\ldots$ run from $1$ to $n$ and the subscripts $i$, $j$, $\ldots$ run from $1$ to $p$.
We denote the Kronecker delta by $\delta_{ij}$.

We employ the following notations for matrix derivatives introduced in \cite{Matsuda22}.

\begin{definition}
	For a function $f: \mathbb{R}^{n \times p} \to \mathbb{R}$, its \textrm{matrix gradient} $\widetilde{\nabla} f: \mathbb{R}^{n \times p} \to \mathbb{R}^{n \times p}$ is defined as 
	\begin{align}
		(\widetilde{\nabla} f(X))_{ai} = \frac{\partial}{\partial X_{ai}} f(X). \label{grad}
	\end{align}
\end{definition}

\begin{definition}
For a $C^2$ function $f: \mathbb{R}^{n \times p} \to \mathbb{R}$, its {matrix Laplacian} $\widetilde{\Delta} f: \mathbb{R}^{n \times p} \to \mathbb{R}^{p \times p}$ is defined as 
\begin{align}
	(\widetilde{\Delta} f(X))_{ij} = \sum_{a=1}^n \frac{\partial^2}{\partial X_{ai} \partial X_{aj}} f(X). \label{lap}
\end{align}
\end{definition}

Let
\begin{align}
	X^{\top} X = V \Lambda V^{\top} \label{XtX}
\end{align}
be a spectral decomposition of $X^{\top} X$, where $V=(v_1,\dots,v_p)$ is an orthogonal matrix and $\Lambda = {\rm diag}(\lambda_1,\dots,\lambda_p)$ is a diagonal matrix.
Then, the derivatives of $\lambda$ and $V$ are obtained as follows.

\begin{lemma}
	The derivative of $\lambda_i$ is
\begin{align}\label{lambda_diff}
	\frac{\partial \lambda_i}{\partial X_{aj}} = 2  V_{ji} \sum_k X_{ak} V_{ki}.
\end{align}
	Thus,
\begin{align}
	\widetilde{\nabla} \lambda_i = 2 X v_i v_i^{\top}, \label{nabla_lambda}
\end{align}
	where $v_i$ is the $i$-th column vector of $V$. \end{lemma}
\begin{proof}
By differentiating $V^{\top} V = I_p$ and using $({\rm d} V)^{\top} V=(V^{\top} {\rm d} V)^{\top}$, we obtain
\begin{align}
	V^{\top} {\rm d} V + (V^{\top} {\rm d} V)^{\top} = O, \label{anti}
\end{align}
which means the antisymmetricity of $V^{\top} {\rm d} V$.

Taking the differential of \eqref{XtX}, we have
\begin{align}
	{\rm d} (X^{\top} X) = ({\rm d} V) \Lambda V^{\top} + V ({\rm d} \Lambda) V^{\top} + V \Lambda ({\rm d} V)^{\top}. \label{dXX}
\end{align}
Then, multiplying \eqref{dXX} on the left by $V^{\top}$ and on the right by $V$, we obtain
\begin{align}
	V^{\top} \cdot {\rm d} (X^{\top} X) \cdot V = V^{\top} ({\rm d} V) \Lambda + {\rm d} \Lambda + \Lambda ({\rm d} V)^{\top} V. \label{VtV}
\end{align}
Since $\Lambda$ and ${\rm d} \Lambda$ are diagonal and $({\rm d} V)^{\top} V=(V^{\top} {\rm d} V)^{\top}$, the $(i,j)$-th entry of \eqref{VtV} yields
\begin{align*}
	(V^{\top} \cdot {\rm d} (X^{\top} X) \cdot V)_{ij} &= (V^{\top} {\rm d} V)_{ij} \lambda_j + \delta_{ij} {\rm d} \lambda_i + \lambda_i (({\rm d} V)^{\top} V)_{ij} \nonumber \\
	&= (V^{\top} {\rm d} V)_{ij} \lambda_j + \delta_{ij} {\rm d} \lambda_i + \lambda_i (V^{\top} {\rm d} V)_{ji}. 
\end{align*}
Since $(V^{\top} {\rm d} V)_{ji}=-(V^{\top} {\rm d} V)_{ij}$ from \eqref{anti}, we obtain
\begin{align}
	(V^{\top} \cdot {\rm d} (X^{\top} X) \cdot V)_{ij} &= (\lambda_j-\lambda_i) (V^{\top} {\rm d} V)_{ij} + \delta_{ij} {\rm d} \lambda_i. \label{ddd}
\end{align}

On the other hand, from ${\rm d} (X^{\top} X) = ({\rm d} X)^{\top} X + X^{\top} {\rm d} X$,
\begin{align}
	{\rm d} (X^{\top} X)_{ij} = \sum_a (({\rm d} X)_{ai} X_{aj} + X_{ai} ({\rm d} X)_{aj}). \label{dXtX}
\end{align}

	By taking $i=j$ in \eqref{ddd},
\begin{align*}
	{\rm d} \lambda_i = (V^{\top} \cdot {\rm d} (X^{\top} X) \cdot V)_{ii} = \sum_{j,k} (V^{\top})_{ij} ({\rm d} (X^{\top} X))_{jk} V_{ki}.
\end{align*}
	Then, by using \eqref{dXtX},
\begin{align*}
	{\rm d} \lambda_i &= \sum_{j,k} V_{ji} V_{ki} \sum_a (({\rm d} X)_{aj} X_{ak} + X_{aj} ({\rm d} X)_{ak}) \\
	&= 2 \sum_{j,k} V_{ji} V_{ki} \sum_a  X_{ak} ({\rm d} X)_{aj} \\
	&= 2 \sum_{a,j} V_{ji} \sum_k  X_{ak} V_{ki}  ({\rm d} X)_{aj}.
\end{align*}
	Thus, we obtain \eqref{lambda_diff} and it leads to \eqref{nabla_lambda}.
\end{proof}

\begin{lemma}
	The derivative of $V_{ij}$ is
	\begin{align}\label{Vdiff}
	\frac{\partial V_{ij}}{\partial X_{ak}} = \sum_{l \neq j} \frac{V_{il}}{\lambda_j-\lambda_l} ((XV)_{aj} V_{kl} + (XV)_{al} V_{kj}).
	\end{align}
\end{lemma}
\begin{proof}
	From \eqref{anti}, we have $(V^{\top} {\rm d} V)_{ii}=0$.
	Also, from \eqref{ddd},
	\begin{align*}
		(V^{\top} {\rm d} V)_{ij} = \frac{1}{\lambda_j-\lambda_i} (V^{\top} \cdot {\rm d} (X^{\top} X) \cdot V)_{ij}
	\end{align*}
	for $i \neq j$.
	Therefore, 
	\begin{align*}
		{\rm d} V_{ij} &= (V V^{\top} {\rm d} V)_{ij} \\
		&= \sum_k V_{ik} (V^{\top} {\rm d} V)_{kj} \\
		&= \sum_{k \neq j} V_{ik} \frac{1}{\lambda_j-\lambda_k} (V^{\top} \cdot {\rm d} (X^{\top} X) \cdot V)_{kj} \\
		&= \sum_{k \neq j} \frac{1}{\lambda_j-\lambda_k} V_{ik} \sum_{l,m} V_{lk} {\rm d} (X^{\top} X)_{lm} V_{mj}.
	\end{align*}
	Then, by using \eqref{dXtX},
	\begin{align*}
	{\rm d} V_{ij} &= \sum_{k \neq j} \frac{1}{\lambda_j-\lambda_k} V_{ik} \sum_{l,m} V_{lk} \sum_a (({\rm d} X)_{al} X_{am} + X_{al} ({\rm d} X)_{am}) V_{mj} \\
	&= \sum_{k \neq j} \frac{1}{\lambda_j-\lambda_k} V_{ik} \sum_a \left(  (XV)_{aj} \sum_l V_{lk} ({\rm d} X)_{al} + (XV)_{ak} \sum_m V_{mj} ({\rm d} X)_{am} \right) \\
	&= \sum_l \sum_{k \neq j} \frac{V_{ik}}{\lambda_j-\lambda_k} \sum_a ((XV)_{aj} V_{lk}+(XV)_{ak} V_{lj}) ({\rm d} X)_{al} \\
	&= \sum_{a,k} \sum_{l \neq j} \frac{V_{il}}{\lambda_j-\lambda_l} ((XV)_{aj} V_{kl}+(XV)_{al} V_{kj}) ({\rm d} X)_{ak},
\end{align*}
	where we switched $k$ and $l$ in the last step.
	Thus, we obtain \eqref{Vdiff}.
\end{proof}

A function $h$ is said to be orthogonally invariant if it satisfies $h(PXQ)=h(X)$ for any orthogonal matrices $P \in O(n)$ and $Q \in O(p)$.
Such a function can be written as $h(X)=H(\lambda)$, where $\lambda$ is the eigenvalues of $X^{\top} X$ as given by \eqref{XtX}, and its derivatives are calculated as follows.

\begin{lemma}
	The matrix gradient \eqref{grad} of an orthogonally invariant function $h(X)=H(\lambda)$ is
	\begin{align}\label{nabla_h}
	\widetilde{\nabla} h = 2 \sum_i \frac{\partial H}{\partial \lambda_i} X v_i v_i^{\top}.
	\end{align}
	Thus,
	\begin{align}\label{nabla_h2}
	(\widetilde{\nabla} h)^{\top} (\widetilde{\nabla} h) = 4 \sum_i \left( \frac{\partial H}{\partial \lambda_i} \right)^2 \lambda_i v_i v_i^{\top} = V D V^{\top},
	\end{align}
	where $D$ is the $p \times p$ diagonal matrix given by
	\begin{align*}
	D_{kk} = 4 \lambda_k \left( \frac{\partial H}{\partial \lambda_k} \right)^2.
	\end{align*}
\end{lemma}
\begin{proof}
	From \eqref{nabla_lambda},
	\begin{align*}
\widetilde{\nabla} h = \sum_i \frac{\partial H}{\partial \lambda_i} \widetilde{\nabla} \lambda_i = 2 \sum_i \frac{\partial H}{\partial \lambda_i} X v_i v_i^{\top},
	\end{align*}
	which yields \eqref{nabla_h}.
	Then, by using $X^{\top}X = V \Lambda V^{\top}$ and $V^{\top} V = I_p$,
	\begin{align*}
	(\widetilde{\nabla} h)^{\top} (\widetilde{\nabla} h) &= 4 \sum_{i,j} \frac{\partial H}{\partial \lambda_i} \frac{\partial H}{\partial \lambda_j} v_i v_i^{\top} X^{\top} X v_j v_j^{\top} \\
	&= 4 \sum_{i,j} \frac{\partial H}{\partial \lambda_i} \frac{\partial H}{\partial \lambda_j} v_i \Lambda_{ij} v_j^{\top} \\
	&= 4 \sum_{i} \left( \frac{\partial H}{\partial \lambda_i} \right)^2 \lambda_i v_i v_i^{\top},
	\end{align*}
	which yields \eqref{nabla_h2}.
\end{proof}

\begin{lemma}
	The matrix Laplacian \eqref{lap} of an orthogonally invariant function $h(X)=H(\lambda)$ is
\begin{align}
	\widetilde{\Delta} h = V D V^{\top}, \label{mLap}
\end{align}
	where $D$ is the $p \times p$ diagonal matrix given by
\begin{align*}
	D_{kk} = 4 \lambda_k \frac{\partial^2 H}{\partial \lambda_k^2} + 2n \frac{\partial H}{\partial \lambda_k} + 2 \sum_{l \neq k} \frac{\lambda_l}{\lambda_k-\lambda_l} \left( \frac{\partial H}{\partial \lambda_k} - \frac{\partial H}{\partial \lambda_l} \right).
\end{align*}
\end{lemma}
\begin{proof}
	From \eqref{nabla_h},
\begin{align}
	\frac{\partial^2 h}{\partial X_{ai} \partial X_{aj}} &= \frac{\partial}{\partial X_{ai}} \left( 2 \sum_k \frac{\partial H}{\partial \lambda_k} X v_k v_k^{\top} \right)_{aj} \nonumber \\
	&= 2 \sum_k \left( \left( \widetilde{\nabla} \frac{\partial H}{\partial \lambda_k} \right)_{ai} (X v_k v_k^{\top})_{aj} + \frac{\partial H}{\partial \lambda_k} (v_k v_k^{\top})_{ij} + \frac{\partial H}{\partial \lambda_k} \sum_l X_{al} \frac{\partial}{\partial X_{ai}} (v_k v_k^{\top})_{lj} \right) \nonumber \\
	&= 2 \sum_k \left( 2 \sum_l \frac{\partial^2 H}{\partial \lambda_k \partial \lambda_l} (X v_l v_l^{\top})_{ai} (X v_k v_k^{\top})_{aj} + \frac{\partial H}{\partial \lambda_k} V_{ik} V_{jk} \right. \nonumber \\
	&  \qquad \qquad \qquad \qquad \qquad \left. + \frac{\partial H}{\partial \lambda_k} \sum_l  X_{al} \left( \frac{\partial V_{lk}}{\partial X_{ai}} V_{jk} + V_{lk} \frac{\partial V_{jk}}{\partial X_{ai}} \right) \right). \label{h2diff}
\end{align}
	Also, from \eqref{Vdiff},
\begin{align}
	\frac{\partial V_{lk}}{\partial X_{ai}} V_{jk} + V_{lk} \frac{\partial V_{jk}}{\partial X_{ai}} &= \sum_{m \neq k} \frac{V_{lm}}{\lambda_k-\lambda_m} ((XV)_{ak}V_{im}+(XV)_{am}V_{ik}) V_{jk} \nonumber \\
	& \qquad + V_{lk} \sum_{m \neq k} \frac{V_{jm}}{\lambda_k-\lambda_m} ((XV)_{ak}V_{im}+(XV)_{am}V_{ik}) \nonumber \\
	&= \sum_{m \neq k} \frac{(V_{lm}V_{jk}+V_{lk}V_{jm})((XV)_{ak}V_{im}+(XV)_{am}V_{ik})}{\lambda_k-\lambda_m}. \label{VV}
\end{align}
	By substituting \eqref{VV} into \eqref{h2diff} and taking the sum,
\begin{align*}
	(\widetilde{\Delta} h)_{ij} &= \sum_{a=1}^n \frac{\partial^2 h}{\partial X_{ai} \partial X_{aj}} \\
	&= 4 \sum_{k,l} \frac{\partial^2 H}{\partial \lambda_k \partial \lambda_l} (v_l v_l^{\top} X^{\top} X v_k v_k^{\top})_{ij} 	+ 2n \sum_k \frac{\partial H}{\partial \lambda_k} V_{ik} V_{jk} \\
	& \qquad + 2 \sum_{k,l} \frac{\partial H}{\partial \lambda_k} \sum_{m \neq k} \frac{(V_{lm}V_{jk}+V_{lk}V_{jm})((X^{\top}XV)_{lk}V_{im}+(X^{\top}XV)_{lm}V_{ik})}{\lambda_k-\lambda_m}  \\
	&= \sum_k \left( 4 \lambda_k \frac{\partial^2 H}{\partial \lambda_k^2} + 2n \frac{\partial H}{\partial \lambda_k} \right)V_{ik} V_{jk} \\
	& \qquad + 2 \sum_{k,l} \frac{\partial H}{\partial \lambda_k} \sum_{m \neq k} \frac{(V_{lm}V_{jk}+V_{lk}V_{jm})(\lambda_k V_{lk}V_{im}+\lambda_m V_{lm}V_{ik})}{\lambda_k-\lambda_m}  \\
	&= \sum_k \left( 4 \lambda_k \frac{\partial^2 H}{\partial \lambda_k^2} + 2n \frac{\partial H}{\partial \lambda_k} \right)V_{ik} V_{jk} + 2 \sum_k \frac{\partial H}{\partial \lambda_k} \sum_{m \neq k} \frac{\lambda_k V_{jm} V_{im} + \lambda_m V_{jk} V_{ik}}{\lambda_k-\lambda_m}  \\
\end{align*}
	where we used $X^{\top}X = V \Lambda V^{\top}$ and
\begin{align*}
	&\sum_l {(V_{lm}V_{jk}+V_{lk}V_{jm})(\lambda_k V_{lk}V_{im}+\lambda_m V_{lm}V_{ik})} \\
	=& \lambda_k (\delta_{km} V_{jk} + V_{jm}) V_{im} + \lambda_m (V_{jk} + \delta_{km} V_{jm}) V_{ik}.
\end{align*}
	Then,
\begin{align*}
	(\widetilde{\Delta} h)_{ij} &= \sum_k \left( 4 \lambda_k \frac{\partial^2 H}{\partial \lambda_k^2} + 2n \frac{\partial H}{\partial \lambda_k} + 2 \frac{\partial H}{\partial \lambda_k} \sum_{m \neq k} \frac{\lambda_m}{\lambda_k-\lambda_m} \right)V_{ik} V_{jk} + 2 \sum_k \lambda_k \frac{\partial H}{\partial \lambda_k} \sum_{m \neq k} \frac{V_{im} V_{jm}}{\lambda_k-\lambda_m}  \\
	&= \sum_k \left( 4 \lambda_k \frac{\partial^2 H}{\partial \lambda_k^2} + 2n \frac{\partial H}{\partial \lambda_k} + 2 \sum_{m \neq k} \frac{\lambda_m}{\lambda_k-\lambda_m} \left( \frac{\partial H}{\partial \lambda_k} - \frac{\partial H}{\partial \lambda_m} \right) \right) V_{ik} V_{jk},
\end{align*}
	where we used
\begin{align*}
	\sum_k \lambda_k \frac{\partial H}{\partial \lambda_k} \sum_{m \neq k} \frac{V_{im} V_{jm}}{\lambda_k-\lambda_m} = \sum_m V_{im} V_{jm}  \sum_{k \neq m} \frac{\lambda_k}{\lambda_k-\lambda_m} \frac{\partial H}{\partial \lambda_k} = \sum_k V_{ik} V_{jk}  \sum_{m \neq k} \frac{\lambda_m}{\lambda_m-\lambda_k} \frac{\partial H}{\partial \lambda_m}.
\end{align*}
	Thus, by rewriting $m$ to $l$, we obtain \eqref{mLap}.
\end{proof}

By taking the trace of the matrix Laplacian \eqref{mLap}, we have
\begin{align*}
	{\Delta} h &= {\rm tr} (\widetilde{\Delta} h) = {\rm tr} (D) \\
	&= \sum_k \left( 4 \lambda_k \frac{\partial^2 H}{\partial \lambda_k^2} + 2n \frac{\partial H}{\partial \lambda_k} + 2 \sum_{l \neq k} \frac{\lambda_l}{\lambda_k-\lambda_l} \left( \frac{\partial H}{\partial \lambda_k} - \frac{\partial H}{\partial \lambda_l} \right) \right) \\
	&= \sum_k \left( 4 \lambda_k \frac{\partial^2 H}{\partial \lambda_k^2} + 2n \frac{\partial H}{\partial \lambda_k} \right) + 2 \sum_k \frac{\partial H}{\partial \lambda_k} \sum_{l \neq k} \frac{\lambda_k+\lambda_l}{\lambda_k-\lambda_l} \\
	&= 4 \sum_k \lambda_k \frac{\partial^2 H}{\partial \lambda_k^2} + 2 \sum_k \left( n-p+1 + 2 \lambda_k \sum_{l \neq k} \frac{1}{\lambda_k-\lambda_l} \right) \frac{\partial H}{\partial \lambda_k},
\end{align*}
where we used
\begin{align}
	\sum_{l \neq k} \frac{\lambda_k+\lambda_l}{\lambda_k-\lambda_l} &=  \lambda_k \sum_{l \neq k} \frac{1}{\lambda_k-\lambda_l} + \sum_{l \neq k} \frac{\lambda_l}{\lambda_k-\lambda_l} \nonumber \\
	&=  \lambda_k \sum_{l \neq k} \frac{1}{\lambda_k-\lambda_l} + \sum_{l \neq k} \left( \frac{\lambda_k}{\lambda_k-\lambda_l} - 1 \right) \nonumber \\
	&=  2 \lambda_k \sum_{l \neq k} \frac{1}{\lambda_k-\lambda_l} -p+1.  \label{lambda_sum}
\end{align}
This coincides with the Laplacian formula in \cite{Stein74}.

\section{Risk formula}\label{sec:main}
Now, we derive a general formula for the matrix quadratic risk of orthogonally invariant estimators of the form \eqref{pseudo}.

\begin{theorem}\label{th_main}
	Let $h(X)=H(\lambda)$ be an orthogonally invariant function.
	Then, the matrix quadratic risk of an estimator $\hat{M}=X+\widetilde{\nabla}h(X)$ is given by
\begin{align}
	{\rm E}_M (\hat{M}-M)^{\top} (\hat{M}-M) &= n I_p + {\rm E}_M [ V D V^{\top} ], \label{qrisk}
\end{align}
	where $D$ is the $p \times p$ diagonal matrix given by
\begin{align*}
	D_{kk} &= 4 \left( 2 \lambda_k \frac{\partial^2 H}{\partial \lambda_k^2} + n \frac{\partial H}{\partial \lambda_k} + \lambda_k \left( \frac{\partial H}{\partial \lambda_k} \right)^2 + \sum_{l \neq k} \frac{\lambda_l}{\lambda_k-\lambda_l} \left( \frac{\partial H}{\partial \lambda_k} - \frac{\partial H}{\partial \lambda_l} \right) \right).
\end{align*}
\end{theorem}
\begin{proof}
	From \cite{Matsuda22}, the matrix quadratic risk of an estimator $\hat{M}=X+g(X)$ with a weakly differentiable function $g$ is 
\begin{align}\label{sure}
	{\rm E}_M (\hat{M}-M)^{\top} (\hat{M}-M) = n I_p + {\rm E}_M [\widetilde{\mathrm{div}} \ g(X) + (\widetilde{\mathrm{div}} \ g(X) )^{\top} + g(X)^{\top} g(X)],
\end{align}
	where the {matrix divergence} $\widetilde{\mathrm{div}} \ g: \mathbb{R}^{n \times p} \to \mathbb{R}^{p \times p}$ of a function $g: \mathbb{R}^{n \times p} \to \mathbb{R}^{n \times p}$ is defined as 
	\begin{align*}
		(\widetilde{\mathrm{div}} \ g(X))_{ij} = \sum_{a=1}^n \frac{\partial}{\partial X_{ai}} g_{aj}(X).
	\end{align*}
	Therefore, by substituting $g(X)=\widetilde{\nabla}h(X)$ and using $\widetilde{\mathrm{div}} \circ \widetilde{\nabla}=\widetilde{\Delta}$,
\begin{align*}
	{\rm E}_M (\hat{M}-M)^{\top} (\hat{M}-M) = n I_p + E_M [2 \widetilde{\Delta} h(X) + \widetilde{\nabla}h(X)^{\top} \widetilde{\nabla}h(X)].
\end{align*}
	Thus, by using \eqref{nabla_h2} and \eqref{mLap}, we obtain \eqref{qrisk}.
\end{proof}

By taking the trace of \eqref{qrisk} and using \eqref{lambda_sum}, we obtain the following formula for the Frobenius risk of orthogonally invariant estimators, which coincides with the one given by \cite{Stein74}.

\begin{corollary}
	Let $h(X)=H(\lambda)$ be an orthogonally invariant function.
	Then, the Frobenius risk of an estimator $\hat{M}=X+\widetilde{\nabla}h(X)$ is given by
\begin{align*}
	&{\rm E}_M \| \hat{M}-M \|_{\mathrm{F}}^2 \\
	=& np + 4 {\rm E}_M \left[ \sum_k \left( 2 \lambda_k \frac{\partial^2 H}{\partial \lambda_k^2} + n \frac{\partial H}{\partial \lambda_k} + \lambda_k \left( \frac{\partial H}{\partial \lambda_k} \right)^2 \right) + \sum_k \frac{\partial H}{\partial \lambda_k} \sum_{l \neq k} \frac{\lambda_k + \lambda_l}{\lambda_k-\lambda_l} \right] \\
	=& np + 4 {\rm E}_M \left[ \sum_k \left( 2 \lambda_k \frac{\partial^2 H}{\partial \lambda_k^2} + n \frac{\partial H}{\partial \lambda_k} + \lambda_k \left( \frac{\partial H}{\partial \lambda_k} \right)^2 \right) + \sum_k \frac{\partial H}{\partial \lambda_k} \left( 2 \lambda_k \sum_{l \neq k} \frac{1}{\lambda_k-\lambda_l} -p+1 \right) \right].
\end{align*}
\end{corollary}

We derived the risk formula for orthogonally invariant estimators of the form \eqref{pseudo}, which are called pseudo-Bayes estimators \cite{shr_book}.
The class of pseudo-Bayes estimators includes all Bayes and generalized Bayes estimators.
It is an interesting future work to extend the current result to general orthogonally invariant estimators.
Also, extension to unknown covariance case is an important future problem.	
Note that Section 6.6.2 of \cite{Tsukuma} derived a risk formula for a class of estimators in the unknown covariance setting.

\section{Example}\label{sec:example}
We provide an example of the application of Theorem~\ref{th_main}.
Let $X = U \Sigma V^{\top}$ with $U \in \mathbb{R}^{n \times p}$, $\Sigma = {\rm diag} (\sigma_1, \ldots, \sigma_p)$ and $V \in \mathbb{R}^{p \times p}$ be a singular value decomposition of $X$, where $U^{\top} U = V^{\top} V = I_p$ and $\sigma_1 \geq \cdots \geq \sigma_p \geq 0$ are the singular values of $X$.
We consider an orthogonally invariant estimator given by
\begin{align}
	\hat{M} = U \cdot {\rm diag} \left( \sigma_1 - \frac{c_1}{\sigma_1}, \dots, \sigma_p - \frac{c_p}{\sigma_p} \right) \cdot V^{\top}, \label{shr}
\end{align}
where $c_1,\dots,c_p \geq 0$.

\begin{lemma}
The estimator \eqref{shr} can be written in the form \eqref{pseudo} with 
\begin{align*}
	h(X) = -\frac{1}{2} \sum_{k=1}^p {c_k} \log \lambda_k = - \sum_{k=1}^p {c_k} \log \sigma_k,
\end{align*}
where $\lambda_1,\dots,\lambda_p$ are the eigenvalues of $X^{\top} X$ as shown in \eqref{XtX}.
\end{lemma}
\begin{proof}
	From \eqref{nabla_lambda},
\begin{align*}
	\widetilde{\nabla} \log \lambda_k = \frac{1}{\lambda_k} \widetilde{\nabla} \lambda_k = \frac{2}{\lambda_k} X v_k v_k^{\top}.
\end{align*}
	Thus,
\begin{align*}
	\widetilde{\nabla} h(X) = -X \sum_k \frac{c_k}{\lambda_k} v_k v_k^{\top} = -U \cdot {\rm diag} \left( \frac{c_1}{\sigma_1}, \dots, \frac{c_p}{\sigma_p} \right) \cdot V^{\top},
\end{align*}
	where we used $X=U \Sigma V^{\top}$ and $\lambda_k=\sigma_k^2$.
	Therefore, the estimator \eqref{shr} is written as $\hat{M}=X+\widetilde{\nabla} h(X)$.
\end{proof}

\begin{theorem}\label{th_shr}
	The matrix quadratic risk of the estimator \eqref{shr} is given by
	\begin{align}
		{\rm E}_M (\hat{M}-M)^{\top} (\hat{M}-M) &= n I_p + {\rm E}_M [ V D V^{\top} ], \label{qrisk_stein}
	\end{align}
	where $D$ is the $p \times p$ diagonal matrix given by
	\begin{align*}
	D_{kk} = \frac{1}{\lambda_k} c_k (c_k-2n+4) - \frac{2}{\lambda_k} \sum_{l \neq k} \frac{c_k \lambda_l - c_l \lambda_k}{\lambda_k-\lambda_l}.
\end{align*}
\end{theorem}
\begin{proof}
To apply Theorem~\ref{th_main}, let
\begin{align*}
	H(\lambda) = -\frac{1}{2} \sum_k c_k \log \lambda_k.
\end{align*}
We have
\begin{align*}
	\frac{\partial H}{\partial \lambda_k} = -\frac{c_k}{2 \lambda_k}, \quad \frac{\partial^2 H}{\partial \lambda_k^2} = \frac{c_k}{2 \lambda_k^2}.
\end{align*}
Thus,
\begin{align*}
	D_{kk} &= 4 \left( 2 \lambda_k \frac{\partial^2 H}{\partial \lambda_k^2} + n \frac{\partial H}{\partial \lambda_k} + \lambda_k \left( \frac{\partial H}{\partial \lambda_k} \right)^2 + \sum_{l \neq k} \frac{\lambda_l}{\lambda_k-\lambda_l} \left( \frac{\partial H}{\partial \lambda_k} - \frac{\partial H}{\partial \lambda_l} \right) \right) \\
	&= \frac{1}{\lambda_k} c_k (c_k-2n+4) -2 \sum_{l \neq k} \frac{\lambda_l}{\lambda_k-\lambda_l} \left( \frac{c_k}{\lambda_k} - \frac{c_l}{\lambda_l} \right) \\
	&= \frac{1}{\lambda_k} c_k (c_k-2n+4) - \frac{2}{\lambda_k} \sum_{l \neq k} \frac{c_k \lambda_l - c_l \lambda_k}{\lambda_k-\lambda_l}.
\end{align*}
Therefore, we obtain \eqref{qrisk_stein} from Theorem~\ref{th_main}.
\end{proof}

The Efron--Morris estimator \cite{Efron72} corresponds to \eqref{shr} with $c_k \equiv n-p-1$.
In this case,
\begin{align*}
	D_{kk} &= \frac{1}{\lambda_k} (n-p-1) (-n-p+3) - \frac{2}{\lambda_k} (n-p-1) \sum_{l \neq k} \frac{\lambda_l - \lambda_k}{\lambda_k-\lambda_l} \\
	&= \frac{1}{\lambda_k} (n-p-1) (-n-p+3) + \frac{2}{\lambda_k} (n-p-1) (p-1) \\
	& = - \frac{1}{\lambda_k} (n-p-1)^2.
\end{align*}
Thus, its matrix quadratic risk \eqref{qrisk_stein} is
\begin{align}
	{\rm E}_M (\hat{M}-M)^{\top} (\hat{M}-M) &= n I_p - (n-p-1)^2 {\rm E}_M \left[ (X^{\top} X)^{-1} \right]. \label{EMrisk}
\end{align}
This coincides with the result in \cite{Matsuda22}.

Motivated by Stein's proposal \cite{Stein74} for improving on the Efron--Morris estimator, we consider the estimator \eqref{shr} with $c_k = n+p-2k-1$.
In the following, we call it ``Stein's estimator" for convenience.
Stein \cite{Stein74} stated that the positive-part of Stein's estimator dominates the positive-part of the Efron--Morris estimator under the Frobenius loss\footnote{page 31 of \cite{Stein74}: ``It is not difficult to verify, and follows from the general formula (14) that the estimate (8) is better than the crude Efron--Morris estimate (9)." However, we could not find its proof. It is an interesting future work to fill in this gap.}
, where ``positive-part" means the modification of \eqref{shr} given by
\begin{align}
	\hat{M} = U \cdot {\rm diag} \left( \left(\sigma_1 - \frac{c_1}{\sigma_1} \right)_+, \dots, \left( \sigma_p - \frac{c_p}{\sigma_p} \right)_+ \right) \cdot V^{\top}, \label{positive}
\end{align}
where $(a)_+=\max(0,a)$.
It is known that the estimator \eqref{shr} is dominated by its positive-part \eqref{positive} under the Frobenius loss \cite{Tsukuma08}.

\begin{proposition}\label{prop_stein}
	The matrix quadratic risk of Stein's estimator (estimator \eqref{shr} with $c_k = n+p-2k-1$) is given by
\begin{align*}
	{\rm E}_M (\hat{M}-M)^{\top} (\hat{M}-M) &= n I_p + {\rm E}_M [ V D V^{\top} ],
\end{align*}
where $D$ is the $p \times p$ diagonal matrix given by
\begin{align*}
	D_{kk} &= -\frac{1}{\lambda_k} (n+p-2k-1) (n-3p+2k-1) + 4 \sum_{l \neq k} \frac{k-l}{\lambda_k-\lambda_l}.
\end{align*}
	Thus, Stein's estimator dominates the maximum likelihood estimator under the matrix quadratic loss when $n \geq 3p-1$.
\end{proposition}
\begin{proof}
	By substituting $c_k=n+p-2k-1$ into Theorem~\ref{th_shr},
\begin{align*}
	D_{kk} &= \frac{1}{\lambda_k} (n+p-2k-1) (-n+p-2k+3) - \frac{2}{\lambda_k} \sum_{l \neq k} \frac{(n+p-2k-1) \lambda_l - (n+p-2l-1) \lambda_k}{\lambda_k-\lambda_l} \\
	&= \frac{1}{\lambda_k} (n+p-2k-1) (-n+p-2k+3) - \frac{2}{\lambda_k} \sum_{l \neq k} \left( -\frac{2(k-l) \lambda_k}{\lambda_k-\lambda_l} -(n+p-2k-1) \right) \\
	&= \frac{1}{\lambda_k} (n+p-2k-1) (-n+p-2k+3) + 4 \sum_{l \neq k} \frac{k-l}{\lambda_k-\lambda_l} +\frac{2}{\lambda_k} (p-1)(n+p-2k-1) \\
	&= -\frac{1}{\lambda_k} (n+p-2k-1) (n-3p+2k-1) + 4 \sum_{l \neq k} \frac{k-l}{\lambda_k-\lambda_l}.
\end{align*}
	The second term is nonpositive since $\lambda_1 \geq \lambda_2 \geq \dots \geq \lambda_p$.
	When $n \geq 3p-1$, the first term is also nonpositive and thus
\begin{align*}
	{\rm E}_M (\hat{M}-M)^{\top} (\hat{M}-M) \preceq n I_p.
\end{align*}
\end{proof}

Numerical results indicate that the bound of $n$ in Proposition~\ref{prop_stein} may be relaxed to $n \geq p+2$, which is the same bound with the Efron--Morris estimator.
See Appendix.

Finally, we present simulation results to compare Stein's estimator and the Efron--Morris estimator.

Figure~\ref{fig1} compares the Frobenius risk of Stein's estimator and the Efron--Morris estimator when $n=10$ and $p=3$.
It implies that Stein's estimator dominates the Efron--Morris estimator under the Frobenius loss.
Both estimators attain constant risk reduction when some singular values of $M$ are small, regardless of the magnitude of the other singular values.
Thus, both estimators work well for low rank matrices.
See \cite{Matsuda22} for related discussions.

Figure~\ref{fig2} plots the three eigenvalues $\lambda_1 \geq \lambda_2 \geq \lambda_3$ of the matrix quadratic risk of Stein's estimator and the Efron--Morris estimator in the same setting with Figure~\ref{fig1}.
Since all eigenvalues are less than $n=10$, the matrix quadratic risk $R(M,\hat{M})$ satisfies $R(M,\hat{M}) \preceq n I_p$ for every $M$.
Thus, both estimators dominate the maximum likelihood estimator under the matrix quadratic loss, which is compatible with \eqref{EMrisk} and Proposition~\ref{prop_stein}.
Also, each eigenvalue for Stein's estimator is smaller than the corresponding one for the Efron--Morris estimator, which suggests that Stein's estimator dominates the Efron--Morris estimator even under the matrix quadratic loss.
It is an interesting future work to develop its rigorous theory.

\begin{figure}[h]
	\centering
	\begin{tikzpicture}
		\begin{axis}[
			title={$\sigma_2(M)=\sigma_3(M)=0$},
			xlabel={$\sigma_1(M)$}, xmin=0, xmax=20,
			ylabel={risk}, ymin=0, ymax=31,
			width=0.45\linewidth
			]
			\addplot[thick, color=black,
			filter discard warning=false, unbounded coords=discard
			] table {
				0    7.6561
				1.0000    8.3253
				2.0000   10.0204
				3.0000   11.8952
				4.0000   13.4096
				5.0000   14.4254
				6.0000   15.0407
				7.0000   15.4584
				8.0000   15.6807
				9.0000   15.8459
				10.0000   15.9756
				11.0000   16.0724
				12.0000   16.1727
				13.0000   16.2003
				14.0000   16.2585
				15.0000   16.3073
				16.0000   16.3060
				17.0000   16.3464
				18.0000   16.4030
				19.0000   16.4239
				20.0000   16.4070
			};
			\addplot[dashed, thick, color=black,
			filter discard warning=false, unbounded coords=discard
			] table {
				0   12.0064
				1.0000   12.5152
				2.0000   13.7815
				3.0000   15.0060
				4.0000   15.8875
				5.0000   16.4985
				6.0000   16.8850
				7.0000   17.1756
				8.0000   17.3153
				9.0000   17.4475
				10.0000   17.5487
				11.0000   17.6208
				12.0000   17.7146
				13.0000   17.7286
				14.0000   17.7761
				15.0000   17.8253
				16.0000   17.8041
				17.0000   17.8483
				18.0000   17.8985
				19.0000   17.9192
				20.0000   17.8934
			};
		\end{axis}
	\end{tikzpicture} 
	\begin{tikzpicture}
		\begin{axis}[
			title={$\sigma_1(M)=20, \ \sigma_3(M)=0$},
			xlabel={$\sigma_2(M)$}, xmin=0, xmax=20,
			ylabel={risk}, ymin=0, ymax=31,
			width=0.45\linewidth
			]
			\addplot[thick, color=black,
			filter discard warning=false, unbounded coords=discard
			] table {
				0   16.3957
				1.0000   17.1071
				2.0000   18.6795
				3.0000   20.2608
				4.0000   21.4868
				5.0000   22.2461
				6.0000   22.7175
				7.0000   22.9768
				8.0000   23.1834
				9.0000   23.2943
				10.0000   23.3902
				11.0000   23.4709
				12.0000   23.5347
				13.0000   23.5988
				14.0000   23.6282
				15.0000   23.6586
				16.0000   23.6236
				17.0000   23.6407
				18.0000   23.6504
				19.0000   23.6430
				20.0000   23.5929
			};
			\addplot[dashed, thick, color=black,
			filter discard warning=false, unbounded coords=discard
			] table {
				0   17.8886
				1.0000   18.5126
				2.0000   19.8290
				3.0000   21.0409
				4.0000   21.9492
				5.0000   22.5101
				6.0000   22.8904
				7.0000   23.0995
				8.0000   23.2832
				9.0000   23.3810
				10.0000   23.4595
				11.0000   23.5352
				12.0000   23.5990
				13.0000   23.6604
				14.0000   23.6928
				15.0000   23.7257
				16.0000   23.7010
				17.0000   23.7378
				18.0000   23.7836
				19.0000   23.8171
				20.0000   23.7831
			};
		\end{axis}
	\end{tikzpicture} 
	\caption{Frobenius risk of the Efron--Morris estimator (dashed) and Stein's estimator (solid) for $n=10$ and $p=3$. Left: $\sigma_2(M)=\sigma_3(M)=0$. Right: $\sigma_1(M)=20$, $\sigma_3(M)=0$.}
	\label{fig1}
\end{figure}
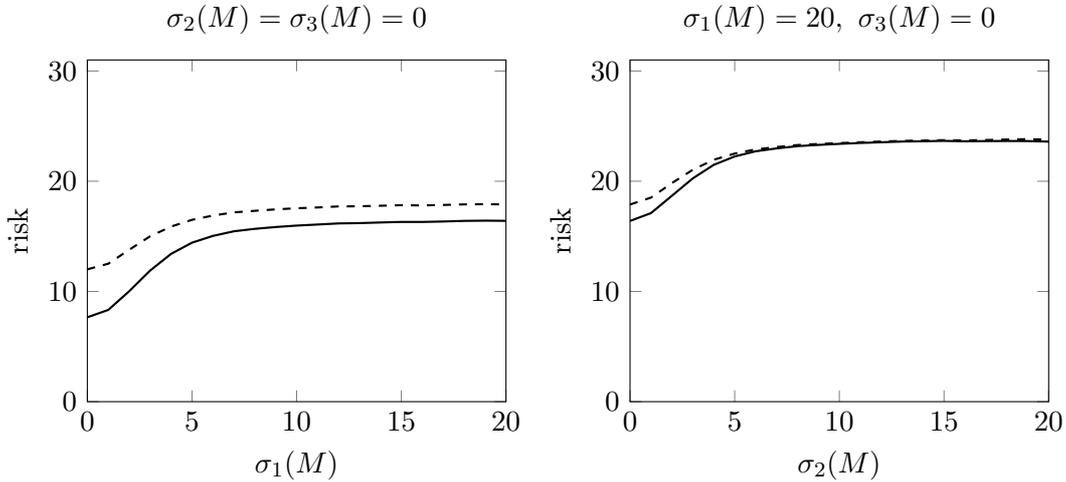

\begin{figure}[h]
	\centering
	\begin{tikzpicture}
		\begin{axis}[
			title={$\sigma_2(M)=\sigma_3(M)=0$},
			xlabel={$\sigma_1(M)$}, xmin=0, xmax=20,
			ylabel={eigenvalue}, ymin=0, ymax=11,
			width=0.45\linewidth
			]
			\addplot[thick, color=black,
			filter discard warning=false, unbounded coords=discard
			] table {
				0    2.5344
				1.0000    2.5454
				2.0000    2.6200
				3.0000    2.7679
				4.0000    2.9289
				5.0000    3.0208
				6.0000    3.1003
				7.0000    3.1484
				8.0000    3.1707
				9.0000    3.1963
				10.0000    3.1888
				11.0000    3.2131
				12.0000    3.2268
				13.0000    3.2288
				14.0000    3.2332
				15.0000    3.2444
				16.0000    3.2380
				17.0000    3.2377
				18.0000    3.2541
				19.0000    3.2622
				20.0000    3.2480
			};
			\addplot[thick, color=black,
			filter discard warning=false, unbounded coords=discard
			] table {
				0    2.5531
				1.0000    2.5734
				2.0000    2.6474
				3.0000    2.7968
				4.0000    2.9370
				5.0000    3.0497
				6.0000    3.1088
				7.0000    3.1679
				8.0000    3.1963
				9.0000    3.2075
				10.0000    3.2279
				11.0000    3.2162
				12.0000    3.2409
				13.0000    3.2383
				14.0000    3.2557
				15.0000    3.2598
				16.0000    3.2474
				17.0000    3.2818
				18.0000    3.2779
				19.0000    3.2684
				20.0000    3.2629
			};
			\addplot[thick, color=black,
			filter discard warning=false, unbounded coords=discard
			] table {
				0    2.5686
				1.0000    3.2065
				2.0000    4.7530
				3.0000    6.3304
				4.0000    7.5436
				5.0000    8.3549
				6.0000    8.8316
				7.0000    9.1421
				8.0000    9.3138
				9.0000    9.4421
				10.0000    9.5589
				11.0000    9.6431
				12.0000    9.7050
				13.0000    9.7333
				14.0000    9.7695
				15.0000    9.8031
				16.0000    9.8207
				17.0000    9.8269
				18.0000    9.8711
				19.0000    9.8932
				20.0000    9.8961
			};
			\addplot[dashed, thick, color=black,
			filter discard warning=false, unbounded coords=discard
			] table {
				0    3.9882
				1.0000    3.9768
				2.0000    3.9831
				3.0000    3.9941
				4.0000    4.0045
				5.0000    3.9837
				6.0000    3.9921
				7.0000    4.0005
				8.0000    3.9955
				9.0000    3.9998
				10.0000    3.9753
				11.0000    3.9943
				12.0000    4.0054
				13.0000    3.9974
				14.0000    3.9981
				15.0000    4.0053
				16.0000    3.9906
				17.0000    3.9886
				18.0000    4.0024
				19.0000    4.0124
				20.0000    3.9949
			};
			\addplot[dashed, thick, color=black,
			filter discard warning=false, unbounded coords=discard
			] table {
				0    4.0055
				1.0000    4.0031
				2.0000    4.0070
				3.0000    4.0249
				4.0000    4.0160
				5.0000    4.0222
				6.0000    4.0102
				7.0000    4.0221
				8.0000    4.0165
				9.0000    4.0136
				10.0000    4.0223
				11.0000    3.9954
				12.0000    4.0144
				13.0000    4.0074
				14.0000    4.0205
				15.0000    4.0229
				16.0000    4.0022
				17.0000    4.0382
				18.0000    4.0317
				19.0000    4.0201
				20.0000    4.0076
			};
			\addplot[dashed, thick, color=black,
			filter discard warning=false, unbounded coords=discard
			] table {
				0    4.0127
				1.0000    4.5353
				2.0000    5.7913
				3.0000    6.9870
				4.0000    7.8670
				5.0000    8.4926
				6.0000    8.8827
				7.0000    9.1530
				8.0000    9.3033
				9.0000    9.4341
				10.0000    9.5511
				11.0000    9.6311
				12.0000    9.6948
				13.0000    9.7239
				14.0000    9.7576
				15.0000    9.7971
				16.0000    9.8113
				17.0000    9.8215
				18.0000    9.8645
				19.0000    9.8867
				20.0000    9.8909
			};
		\end{axis}
	\end{tikzpicture} 
	\begin{tikzpicture}
		\begin{axis}[
			title={$\sigma_1(M)=20, \ \sigma_3(M)=0$},
			xlabel={$\sigma_2(M)$}, xmin=0, xmax=20,
			ylabel={eigenvalue}, ymin=0, ymax=11,
			width=0.45\linewidth
			]
			\addplot[thick, color=black,
			filter discard warning=false, unbounded coords=discard
			] table {
				0    3.2521
				1.0000    3.2962
				2.0000    3.3713
				3.0000    3.5289
				4.0000    3.7097
				5.0000    3.8018
				6.0000    3.8712
				7.0000    3.8815
				8.0000    3.9267
				9.0000    3.9215
				10.0000    3.9573
				11.0000    3.9458
				12.0000    3.9540
				13.0000    3.9465
				14.0000    3.9696
				15.0000    3.9632
				16.0000    3.9636
				17.0000    3.9637
				18.0000    3.9821
				19.0000    3.9936
				20.0000    3.9847
			};
			\addplot[thick, color=black,
			filter discard warning=false, unbounded coords=discard
			] table {
				0    3.2688
				1.0000    3.9190
				2.0000    5.4156
				3.0000    6.8457
				4.0000    7.8787
				5.0000    8.5273
				6.0000    8.9293
				7.0000    9.2001
				8.0000    9.3552
				9.0000    9.4964
				10.0000    9.5656
				11.0000    9.6317
				12.0000    9.6866
				13.0000    9.7614
				14.0000    9.7632
				15.0000    9.7994
				16.0000    9.7768
				17.0000    9.7945
				18.0000    9.8091
				19.0000    9.7868
				20.0000    9.8012
			};
			\addplot[thick, color=black,
			filter discard warning=false, unbounded coords=discard
			] table {
				0    9.8748
				1.0000    9.8920
				2.0000    9.8926
				3.0000    9.8861
				4.0000    9.8983
				5.0000    9.9170
				6.0000    9.9171
				7.0000    9.8952
				8.0000    9.9015
				9.0000    9.8764
				10.0000    9.8672
				11.0000    9.8934
				12.0000    9.8941
				13.0000    9.8910
				14.0000    9.8954
				15.0000    9.8960
				16.0000    9.8832
				17.0000    9.8825
				18.0000    9.8592
				19.0000    9.8627
				20.0000    9.8070
			};
			\addplot[dashed, thick, color=black,
			filter discard warning=false, unbounded coords=discard
			] table {
				0    4.0022
				1.0000    4.0170
				2.0000    3.9911
				3.0000    3.9863
				4.0000    4.0117
				5.0000    3.9981
				6.0000    4.0087
				7.0000    3.9855
				8.0000    4.0109
				9.0000    3.9914
				10.0000    4.0180
				11.0000    3.9990
				12.0000    4.0017
				13.0000    3.9900
				14.0000    4.0100
				15.0000    4.0011
				16.0000    3.9991
				17.0000    3.9976
				18.0000    4.0147
				19.0000    4.0248
				20.0000    4.0144
			};
			\addplot[dashed, thick, color=black,
			filter discard warning=false, unbounded coords=discard
			] table {
				0    4.0174
				1.0000    4.6103
				2.0000    5.9525
				3.0000    7.1753
				4.0000    8.0470
				5.0000    8.6031
				6.0000    8.9734
				7.0000    9.2284
				8.0000    9.3764
				9.0000    9.5167
				10.0000    9.5803
				11.0000    9.6471
				12.0000    9.7045
				13.0000    9.7796
				14.0000    9.7847
				15.0000    9.8234
				16.0000    9.8059
				17.0000    9.8342
				18.0000    9.8632
				19.0000    9.8616
				20.0000    9.8812
			};
			\addplot[dashed, thick, color=black,
			filter discard warning=false, unbounded coords=discard
			] table {
				0    9.8690
				1.0000    9.8853
				2.0000    9.8854
				3.0000    9.8793
				4.0000    9.8905
				5.0000    9.9089
				6.0000    9.9082
				7.0000    9.8855
				8.0000    9.8960
				9.0000    9.8729
				10.0000    9.8613
				11.0000    9.8891
				12.0000    9.8928
				13.0000    9.8908
				14.0000    9.8980
				15.0000    9.9011
				16.0000    9.8961
				17.0000    9.9061
				18.0000    9.9056
				19.0000    9.9307
				20.0000    9.8875
			};
		\end{axis}
	\end{tikzpicture} 
	\caption{Eigenvalues of the matrix quadratic risk of the Efron--Morris estimator (dashed) and Stein's estimator (solid) for $n=10$ and $p=3$. Left: $\sigma_2(M)=\sigma_3(M)=0$. Right: $\sigma_1(M)=20$, $\sigma_3(M)=0$. In the left panel, the second and third eigenvalues of each estimator almost overlap. In the right panel, the first eigenvalues of two estimators almost overlap.}
	\label{fig2}
\end{figure}
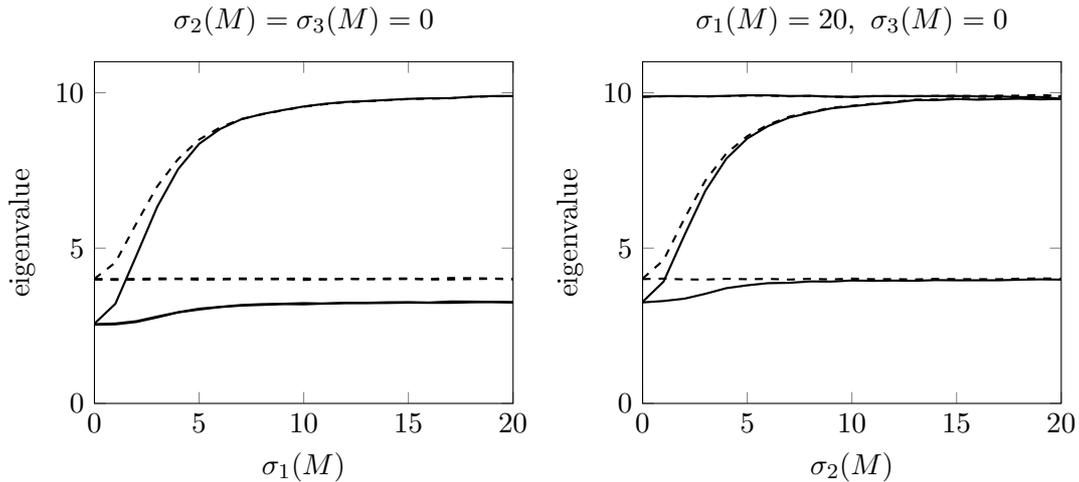

Figures~\ref{fig3} and \ref{fig4} present the results for the positive-part estimators in the same settings with Figures~\ref{fig1} and \ref{fig2}, respectively.
They show qualitatively the same behavior.

\begin{figure}[h]
	\centering
	\begin{tikzpicture}
	\begin{axis}[
		title={$\sigma_2(M)=\sigma_3(M)=0$},
		xlabel={$\sigma_1(M)$}, xmin=0, xmax=20,
		ylabel={risk}, ymin=0, ymax=31,
		width=0.45\linewidth
		]
		\addplot[thick, color=black,
		filter discard warning=false, unbounded coords=discard
		] table {
         0    3.9709
1.0000    4.7983
2.0000    6.8373
3.0000    9.1423
4.0000   10.9578
5.0000   12.1565
6.0000   12.8787
7.0000   13.3255
8.0000   13.6297
9.0000   13.8371
10.0000   13.9842
11.0000   14.1013
12.0000   14.1600
13.0000   14.2149
14.0000   14.2769
15.0000   14.3021
16.0000   14.3347
17.0000   14.3845
18.0000   14.3947
19.0000   14.4115
20.0000   14.4626
		};
		\addplot[dashed, thick, color=black,
		filter discard warning=false, unbounded coords=discard
		] table {
         0    8.6756
1.0000    9.3394
2.0000   10.8799
3.0000   12.4440
4.0000   13.5836
5.0000   14.3309
6.0000   14.8161
7.0000   15.1407
8.0000   15.3763
9.0000   15.5358
10.0000   15.6492
11.0000   15.7383
12.0000   15.7890
13.0000   15.8252
14.0000   15.8845
15.0000   15.8985
16.0000   15.9271
17.0000   15.9752
18.0000   15.9773
19.0000   15.9931
20.0000   16.0401
		};
	\end{axis}
\end{tikzpicture} 
	\begin{tikzpicture}
		\begin{axis}[
			title={$\sigma_1(M)=20, \ \sigma_3(M)=0$},
			xlabel={$\sigma_2(M)$}, xmin=0, xmax=20,
			ylabel={risk}, ymin=0, ymax=31,
			width=0.45\linewidth
			]
			\addplot[thick, color=black,
			filter discard warning=false, unbounded coords=discard
			] table {
         0   14.4602
1.0000   15.2334
2.0000   17.1631
3.0000   19.1134
4.0000   20.5513
5.0000   21.3643
6.0000   21.8847
7.0000   22.2166
8.0000   22.4346
9.0000   22.5602
10.0000   22.6798
11.0000   22.7293
12.0000   22.8334
13.0000   22.8236
14.0000   22.9061
15.0000   22.9194
16.0000   22.9337
17.0000   22.9385
18.0000   22.8915
19.0000   22.8769
20.0000   22.8972
			};
			\addplot[dashed, thick, color=black,
			filter discard warning=false, unbounded coords=discard
			] table {
         0   16.0340
1.0000   16.7217
2.0000   18.3662
3.0000   19.9160
4.0000   21.0179
5.0000   21.6324
6.0000   22.0572
7.0000   22.3384
8.0000   22.5305
9.0000   22.6417
10.0000   22.7524
11.0000   22.7955
12.0000   22.8949
13.0000   22.8812
14.0000   22.9683
15.0000   22.9850
16.0000   23.0135
17.0000   23.0361
18.0000   23.0266
19.0000   23.0496
20.0000   23.0879
			};
		\end{axis}
	\end{tikzpicture} 
	\caption{Frobenius risk of the positive-part Efron--Morris estimator (dashed) and postive-part Stein's estimator (solid) for $n=10$ and $p=3$. Left: $\sigma_2(M)=\sigma_3(M)=0$. Right: $\sigma_1(M)=20$, $\sigma_3(M)=0$.}
	\label{fig3}
\end{figure}

\begin{figure}[h]
	\centering
	\begin{tikzpicture}
		\begin{axis}[
		title={$\sigma_2(M)=\sigma_3(M)=0$},
			xlabel={$\sigma_1(M)$}, xmin=0, xmax=20,
			ylabel={eigenvalue}, ymin=0, ymax=11,
			width=0.45\linewidth
			]
			\addplot[thick, color=black,
			filter discard warning=false, unbounded coords=discard
			] table {
         0    1.3151
1.0000    1.3670
2.0000    1.4940
3.0000    1.6851
4.0000    1.8727
5.0000    2.0189
6.0000    2.0970
7.0000    2.1429
8.0000    2.1847
9.0000    2.2073
10.0000    2.2245
11.0000    2.2336
12.0000    2.2545
13.0000    2.2505
14.0000    2.2663
15.0000    2.2626
16.0000    2.2678
17.0000    2.2723
18.0000    2.2655
19.0000    2.2771
20.0000    2.2852
			};
			\addplot[thick, color=black,
			filter discard warning=false, unbounded coords=discard
			] table {
         0    1.3275
1.0000    1.3823
2.0000    1.5119
3.0000    1.6969
4.0000    1.8860
5.0000    2.0237
6.0000    2.1029
7.0000    2.1625
8.0000    2.1901
9.0000    2.2216
10.0000    2.2413
11.0000    2.2488
12.0000    2.2616
13.0000    2.2591
14.0000    2.2726
15.0000    2.2706
16.0000    2.2746
17.0000    2.2875
18.0000    2.2738
19.0000    2.2838
20.0000    2.3029
			};
			\addplot[thick, color=black,
			filter discard warning=false, unbounded coords=discard
			] table {
         0    1.3283
1.0000    2.0490
2.0000    3.8314
3.0000    5.7604
4.0000    7.1991
5.0000    8.1139
6.0000    8.6787
7.0000    9.0201
8.0000    9.2549
9.0000    9.4083
10.0000    9.5184
11.0000    9.6189
12.0000    9.6438
13.0000    9.7053
14.0000    9.7380
15.0000    9.7689
16.0000    9.7923
17.0000    9.8247
18.0000    9.8554
19.0000    9.8506
20.0000    9.8746
			};
			\addplot[dashed, thick, color=black,
filter discard warning=false, unbounded coords=discard
] table {
         0    2.8774
1.0000    2.9068
2.0000    2.9259
3.0000    2.9688
4.0000    2.9994
5.0000    3.0334
6.0000    3.0356
7.0000    3.0383
8.0000    3.0538
9.0000    3.0521
10.0000    3.0573
11.0000    3.0527
12.0000    3.0687
13.0000    3.0584
14.0000    3.0748
15.0000    3.0644
16.0000    3.0669
17.0000    3.0697
18.0000    3.0605
19.0000    3.0689
20.0000    3.0765
};
\addplot[dashed, thick, color=black,
filter discard warning=false, unbounded coords=discard
] table {

0    2.8972
1.0000    2.9242
2.0000    2.9575
3.0000    2.9846
4.0000    3.0186
5.0000    3.0380
6.0000    3.0428
7.0000    3.0620
8.0000    3.0604
9.0000    3.0746
10.0000    3.0783
11.0000    3.0763
12.0000    3.0850
13.0000    3.0694
14.0000    3.0776
15.0000    3.0727
16.0000    3.0750
17.0000    3.0837
18.0000    3.0673
19.0000    3.0745
20.0000    3.0935
};
\addplot[dashed, thick, color=black,
filter discard warning=false, unbounded coords=discard
] table {
         0    2.9011
1.0000    3.5084
2.0000    4.9965
3.0000    6.4907
4.0000    7.5657
5.0000    8.2596
6.0000    8.7376
7.0000    9.0404
8.0000    9.2621
9.0000    9.4091
10.0000    9.5136
11.0000    9.6094
12.0000    9.6352
13.0000    9.6973
14.0000    9.7322
15.0000    9.7614
16.0000    9.7851
17.0000    9.8219
18.0000    9.8495
19.0000    9.8497
20.0000    9.8700
};
		\end{axis}
	\end{tikzpicture} 
	\begin{tikzpicture}
		\begin{axis}[
			title={$\sigma_1(M)=20, \ \sigma_3(M)=0$},
			xlabel={$\sigma_2(M)$}, xmin=0, xmax=20,
			ylabel={eigenvalue}, ymin=0, ymax=11,
			width=0.45\linewidth
			]
			\addplot[thick, color=black,
filter discard warning=false, unbounded coords=discard
] table {
         0    2.2806
1.0000    2.3366
2.0000    2.5072
3.0000    2.7235
4.0000    2.9231
5.0000    3.0528
6.0000    3.1159
7.0000    3.1797
8.0000    3.1893
9.0000    3.2037
10.0000    3.2301
11.0000    3.2139
12.0000    3.2306
13.0000    3.2421
14.0000    3.2339
15.0000    3.2511
16.0000    3.2586
17.0000    3.2517
18.0000    3.2538
19.0000    3.2518
20.0000    3.2577
};
\addplot[thick, color=black,
filter discard warning=false, unbounded coords=discard
] table {
         0    2.2948
1.0000    3.0352
2.0000    4.7744
3.0000    6.5068
4.0000    7.7223
5.0000    8.4342
6.0000    8.8589
7.0000    9.1452
8.0000    9.3468
9.0000    9.4629
10.0000    9.5493
11.0000    9.6245
12.0000    9.6982
13.0000    9.6845
14.0000    9.7636
15.0000    9.7693
16.0000    9.7929
17.0000    9.8252
18.0000    9.7933
19.0000    9.7994
20.0000    9.8097
};
\addplot[thick, color=black,
filter discard warning=false, unbounded coords=discard
] table {
         0    9.8848
1.0000    9.8617
2.0000    9.8815
3.0000    9.8832
4.0000    9.9059
5.0000    9.8773
6.0000    9.9098
7.0000    9.8917
8.0000    9.8985
9.0000    9.8937
10.0000    9.9004
11.0000    9.8909
12.0000    9.9045
13.0000    9.8969
14.0000    9.9086
15.0000    9.8989
16.0000    9.8821
17.0000    9.8616
18.0000    9.8444
19.0000    9.8257
20.0000    9.8298
};
			\addplot[dashed, thick, color=black,
			filter discard warning=false, unbounded coords=discard
			] table {
         0    3.0707
1.0000    3.0874
2.0000    3.1431
3.0000    3.1854
4.0000    3.2260
5.0000    3.2496
6.0000    3.2528
7.0000    3.2838
8.0000    3.2728
9.0000    3.2738
10.0000    3.2907
11.0000    3.2671
12.0000    3.2787
13.0000    3.2860
14.0000    3.2744
15.0000    3.2889
16.0000    3.2944
17.0000    3.2855
18.0000    3.2862
19.0000    3.2828
20.0000    3.2874
			};
			\addplot[dashed, thick, color=black,
			filter discard warning=false, unbounded coords=discard
			] table {
         0    3.0853
1.0000    3.7749
2.0000    5.3474
3.0000    6.8513
4.0000    7.8931
5.0000    8.5145
6.0000    8.9015
7.0000    9.1723
8.0000    9.3654
9.0000    9.4821
10.0000    9.5666
11.0000    9.6422
12.0000    9.7156
13.0000    9.7010
14.0000    9.7846
15.0000    9.7928
16.0000    9.8239
17.0000    9.8644
18.0000    9.8510
19.0000    9.8718
20.0000    9.8891
			};
			\addplot[dashed, thick, color=black,
			filter discard warning=false, unbounded coords=discard
			] table {
         0    9.8780
1.0000    9.8594
2.0000    9.8757
3.0000    9.8792
4.0000    9.8987
5.0000    9.8683
6.0000    9.9030
7.0000    9.8823
8.0000    9.8924
9.0000    9.8858
10.0000    9.8951
11.0000    9.8862
12.0000    9.9006
13.0000    9.8942
14.0000    9.9093
15.0000    9.9033
16.0000    9.8952
17.0000    9.8862
18.0000    9.8894
19.0000    9.8951
20.0000    9.9115
			};
		\end{axis}
	\end{tikzpicture} 
	\caption{Eigenvalues of the matrix quadratic risk of the positive-part Efron--Morris estimator (dashed) and postive-part Stein's estimator (solid) for $n=10$ and $p=3$. Left: $\sigma_2(M)=\sigma_3(M)=0$. Right: $\sigma_1(M)=20$, $\sigma_3(M)=0$. In the left panel, the second and third eigenvalues of each estimator almost overlap. In the right panel, the first eigenvalues of two estimators almost overlap.}
	\label{fig4}
\end{figure}

\section*{Acknowledgements}
The author thanks the reviewer for constructive comments.
The author thanks William Strawderman for helpful comments.
This work was supported by JSPS KAKENHI Grant Numbers 21H05205, 22K17865 and JST Moonshot Grant Number JPMJMS2024.

\appendix
\section{Bound of $n$ in Proposition~\ref{prop_stein}}
Figure~\ref{fig_appendix} plots the largest eigenvalue of the matrix quadratic risk of Stein's estimator when $\sigma_1(M)=\dots=\sigma_p(M)=50$.
It indicates that the bound of $n$ in Proposition~\ref{prop_stein} may be relaxed to $n \geq p+2$.

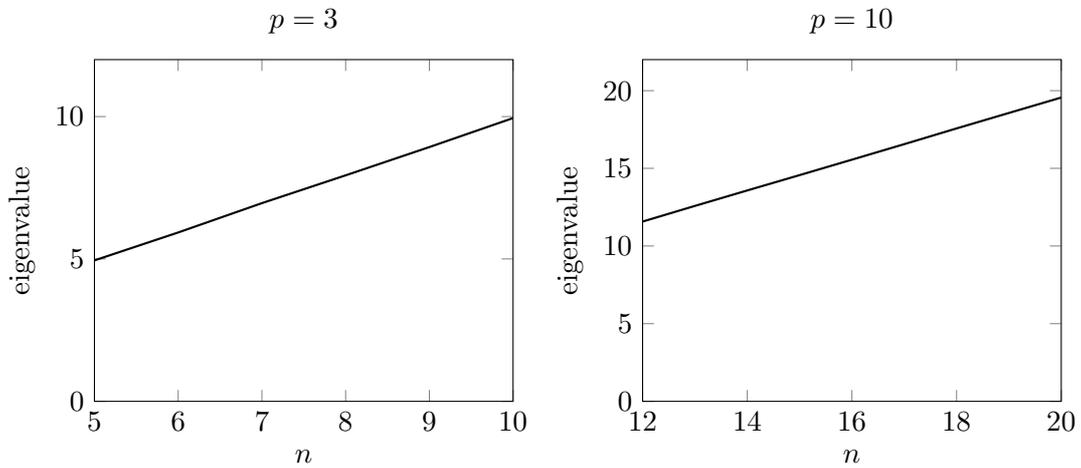
\begin{figure}[h]
	\centering
	\begin{tikzpicture}
		\begin{axis}[
			title={$p=3$},
			xlabel={$n$}, xmin=5, xmax=10,
			ylabel={eigenvalue}, ymin=0, ymax=12,
			width=0.45\linewidth
			]
			\addplot[thick, color=black,
			filter discard warning=false, unbounded coords=discard
			] table {
    5.0000    4.9489
6.0000    5.9295
7.0000    6.9599
8.0000    7.9360
9.0000    8.9289
10.0000    9.9469
			};
		\end{axis}
	\end{tikzpicture} 
	\begin{tikzpicture}
		\begin{axis}[
			title={$p=10$},
			xlabel={$n$}, xmin=12, xmax=20,
			ylabel={eigenvalue}, ymin=0, ymax=22,
			width=0.45\linewidth
			]
			\addplot[thick, color=black,
			filter discard warning=false, unbounded coords=discard
			] table {
   12.0000   11.5728
13.0000   12.5818
14.0000   13.5771
15.0000   14.5643
16.0000   15.5576
17.0000   16.5572
18.0000   17.5705
19.0000   18.5652
20.0000   19.5575
			};
		\end{axis}
	\end{tikzpicture} 
	\caption{Largest eigenvalue of the matrix quadratic risk of Stein's estimator when $\sigma_1(M)=\dots=\sigma_p(M)=50$. Left: $p=3$. Right: $p=10$.}
	\label{fig_appendix}
\end{figure}


\begin{thebibliography}{99}	
	\expandafter\ifx\csname natexlab\endcsname\relax\def\natexlab#1{#1}\fi
	
	\bibitem[{Abu-Shanab et al.(2012)}]{abu}
	\textsc{Abu-Shanab, R.}, \textsc{Kent, J. T.} \& \textsc{Strawderman, W. E.} (2012).
	{Shrinkage estimation with a matrix loss function}.
	\textit{Electronic Journal of Statistics} \textbf{6}, 2347--2355.
	
	\bibitem[{Efron and Morris(1972)}]{Efron72}
	\textsc{Efron, B.} \& \textsc{Morris, C.} (1972).
	{Empirical Bayes on vector observations: an extension of Stein's method}.
	\textit{Biometrika} \textbf{59}, 335--347.
	
	\bibitem[{Fourdrinier et al.(2018)}]{shr_book}
	\textsc{Fourdrinier, D.}, \textsc{Strawderman, W. E.} \& \textsc{Wells, M.} (2018).
	\textit{Shrinkage Estimation}.
	New York: Springer-Verlag.
	
	\bibitem[{Lehmann and Casella(2006)}]{Lehmann}
	\textsc{Lehmann, E. L.} \& \textsc{Casella, G.} (2006).
	\textit{Theory of Point Estimation}.
	New York: Springer-Verlag.
	
	\bibitem[{Matsuda and Komaki(2015)}]{Matsuda}
	\textsc{Matsuda, T.} \& \textsc{Komaki, F.} (2015).
	{Singular value shrinkage priors for Bayesian prediction}.
	\textit{Biometrika} \textbf{102}, 843--854.
	


	\bibitem[{Matsuda and Strawderman(2022)}]{Matsuda22}
\textsc{Matsuda, T.} \& \textsc{Strawderman, W. E.} (2022).
{Estimation under matrix quadratic loss and matrix superharmonicity}.
\textit{Biometrika} \textbf{109}, 503--519.



	\bibitem[{Stein(1974)}]{Stein74}
	\textsc{Stein, C.} (1974).
	{Estimation of the mean of a multivariate normal distribution}.
	\textit{Proc. Prague Symp. Asymptotic Statistics} \textbf{2}, 345--381.

\bibitem[{Tsukuma(2008)}]{Tsukuma08}
\textsc{Tsukuma, H.} (2008).
{Admissibility and minimaxity of Bayes estimators for a normal mean matrix}.
\textit{Journal of Multivariate Analysis} \textbf{99} 2251--2264.

		\bibitem[{Tsukuma and Kubokawa(2020)}]{Tsukuma}
\textsc{Tsukuma, H.} \& \textsc{Kubokawa, T.} (2020).
\textit{Shrinkage estimation for mean and covariance matrices.}.
Springer.

\bibitem[{Yuasa and Kubokawa(2023)}]{Yuasa}
\textsc{Yuasa, R.} \& \textsc{Kubokawa, T.} (2023).
{Generalized Bayes estimators with closed forms for the normal mean and covariance matrices }.
\textit{Journal of Statistical Planning and Inference} \textbf{222} 182--194.

\bibitem[{Yuasa and Kubokawa(2023)}]{Yuasa2}
\textsc{Yuasa, R.} \& \textsc{Kubokawa, T.} (2023).
{Weighted shrinkage estimators of normal mean matrices and dominance properties}.
\textit{Journal of Multivariate Analysis} \textbf{194}, 105138.

\bibitem[{Zheng(1986)}]{Zheng}
\textsc{Zheng, Z.} (1986).
{On estimation of matrix of normal mean}.
\textit{Journal of Multivariate Analysis} \textbf{18} 70--82.

\end{thebibliography}
\end{document}